\documentclass[article]{amsart}
\usepackage{cases}
\usepackage{amsmath, amsfonts, pifont, amssymb}
\usepackage{verbatim}
\newtheorem{theorem}{Theorem}[section]
\newtheorem{claim}[theorem]{Claim}
\newtheorem{lemma}[theorem]{Lemma}
\newtheorem{proposition}[theorem]{Proposition}
\newtheorem{corollary}[theorem]{Corollary}

\newcommand{\R}{\mathbb{R}}

\newcommand{\f}{\frac}

\theoremstyle{definition}
\newtheorem{definition}[theorem]{Definition}

\newcommand{\op}{\operatorname}

\newtheorem{remark}[theorem]{Remark}
\newtheorem{conjecture}[theorem]{{\bfseries Conjecture}}
\numberwithin{equation}{section}

\newcommand{\beq}{\begin{equation}}
\newcommand{\eeq}{\end{equation}}

\begin{document}
\title[]{the interctitical defocusing nonlinear Schr\"odinger equations with radial initial data in dimensions four and higher}

\author{Chuanwei Gao}
\address{The Graduate School of China Academy of Engineering Physics, P. O. Box 2101, Beijing, China, 100088}
\email{canvee@163.com}

\author{Changxing Miao}
\address{Institute of Applied Physics and Computational Mathematics, Beijing 100088}
\email{miao\_changxing@iapcm.ac.cn}

\author{Jianwei Yang}
\address{Beijing International Center for Mathematical Research, Peking University, Beijing 100871, China}
\email{geewey\_{}young@pku.edu.cn}

\begin{abstract}
In this paper, we consider the defocusing nonlinear Schr\"odinger equation in space dimensions $d\geq 4$. We prove that if  $u$ is a radial solution which is \emph{priori} bounded in the critical Sobolev space, that is, $u\in L_t^\infty \dot{H}^{s_c}_x$, then $u$ is global and scatters.   In practise, we use weighted Strichartz space adapted for our setting which ultimately helps us solve the problems in  cases $d\geq 4$ and $0<s_c<\f{1}{2}$. The results in this paper extend  the work of \cite[Comm. in PDEs, 40(2015), 265-308]{M3} to higher dimensions.
  \end{abstract}
\maketitle
\section{ Introduction}
We consider the  Cauchy problem for the nonlinear Schr\"odinger equation (NLS) in  $\mathbb{R}_t\times \mathbb{R}_x^d \, \text{with} \, d\geq 4$:
\begin{align}\label{sch0}
  \begin{cases}
  (i\partial_t+\Delta)u=\mu |u|^pu,\\
  u(0,x)=u_0(x).
  \end{cases}
  \end{align}
  In particular,  we call the equation \eqref{sch0} defocusing, when $\mu=1$, and focusing when $\mu=-1$. In this paper, we are dedicated to dealing with the defocusing case.

The solutions of equation \eqref{sch0} are left invariant by the scaling transformation
\begin{align}\label{scale}
  u(t,x)\mapsto \lambda^{\frac{2}{p}}u(\lambda^2t,\lambda x)
  \end{align}
for $\lambda>0$.  This scaling  invariance defines a notion of criticality. To be more specified,  a direct computation shows that the only homogeneous $L_x^2$-based Sobolev space that is left invariant by  \eqref{scale} is $\dot{H}^{s_c}_x,$  where the critical regularity $s_c$ is given by $s_c:=\f{d}{2}-\f{2}{p}.$
 We call the problem mass-critical for $s_c=0$, energy-critical for $s_c=1$ and intercritical for $0<s_c<1$.
With $s_c=\f{d}{2}-\f{2}{p}$ in mind, we will transfer from $s_c$ to $p$ freely.

\vskip0.15cm
We proceed by make the  notion of solution precise.

\begin{definition}[Strong solution]
  A function $ u : I\times \R^d \rightarrow \mathbb{C} $ on a non-empty time interval $0\in I $ is a strong solution to \eqref{sch0} if it belongs to $C_t\dot{H}_x^{s_c}(K\times \R^d )\cap L_{t,x}^{\f{d+2}{2}p}(K\times \R^d)$ for any compact interval $K\subset I$ and obeys the Duhamel formula
  \begin{align}\label{eq:duhamel}
    u(t)=e^{it\Delta}u_0-i\int_0^te^{i(t-s)\Delta}(|u|^pu)(s)ds
  \end{align}
  for each $t\in I.$ We call $I$ the lifespan of $u.$ We say that $u $ is a maximal-lifespan solution if it cannot be extended to any strictly larger interval. We say $u$  is a global solution  if $I=\R.$
   \end{definition}

 Let  $u$ be a maximal-lifespan solution
  to the problem \eqref{sch0}, a standard technique shows that the  $\|u\|_{L_{t,x}^{\f{d+2}{2}p}(I\times \mathbb{R}^d)}<\infty$  implies scattering. That is $I=\infty$ and there exists $u_{\pm}\in \dot{H}_x^{s_c}(\mathbb{R}^d)$ such that
  \begin{align*}
   \lim_{t\rightarrow \pm \infty}\|u(t)-e^{it\Delta}u_{\pm}\|_{\dot{H}^{s_c}_x(\mathbb R\times\mathbb{R}^d)}=0.
  \end{align*}

 The above fact promotes us to define the notion of  scattering size and blow up  as follows:

\begin{definition}[Scattering size and blow up]
We define the scattering size of a solution $u:I\times \mathbb{R}^d\rightarrow \mathbb{C}$ to \eqref{sch0} by
\begin{align*}  S_I(u):=\iint_{I\times \mathbb{R}^d}|u(t,x)|^{\f{d+2}{2}p}dxdt.
\end{align*}
If there exists $t_0 \in I $~so that $S_{[t_0,\sup I)}(u)=\infty,$ then we say $u$ \emph{blows up} forward in time, correspondingly if  there exists $t_0 \in I $ so that $S_{(\inf I,t_0]}(u)=\infty,$ then we say $u$ \emph{blows up}  backward in time.
\end{definition}
The problem which we concern in this paper can be subsumed into the following conjecture.
\begin{conjecture}\label{conj}
  Let $d\geq 1, p\geq \f{4}{d}$. Assume $u:I\times \R^d \rightarrow  C $  is  a maximal-lifespan solution to \eqref{sch0} such that
  \begin{align}\label{assum}
    u \in L_t^\infty \dot{H}^{s_c}_x(I\times \R^d),
  \end{align}
 then $u$ is global and scatters, with
    \begin{align}
      S_{\mathbb{R}}(u)\leq C(\|u\|_{L_t^\infty\dot{ H}^{s_c}_x})
    \end{align}
    for some function $C:[0,\infty)\rightarrow[0,\infty).$
  \end{conjecture}
  \begin{remark}
    When $s_c=0$ or  $s_c=1$,  \eqref{assum} is true as a direct consequence of conservation law.  In particular, when $s_c=0, u \in L_t^\infty L_x^2$ is guaranteed by  the mass conservation
     \begin{align}
    M[u(t)]=\int_{\mathbb{R}^d}|u(t,x)|^2dx.
  \end{align}
     When $s_c=1, u \in L_t^\infty \dot{H}^1_x$ follows from the energy conservation
     \begin{align}\label{cone}
    E[u(t)]=\int_{\mathbb{R}^d}\f{1}{2}|\nabla u(t,x)|^{2}+\frac{1}{p+2}|u(t,x)|^{p+2} dx.
  \end{align}
    For $s_c \notin \{0,1\}$, \eqref{assum} can not be deduced from any available conserved quantity and it is a natural artificial assumption as a substitution of conservation law.
  \end{remark}

Before addressing our main results, we will make a brief review on the Conjecture \ref{conj}. It is well known that in the critical case, the lifespan of solution depends not only on the Sobolev norm but also the profile of the initial data, thus the fact that \eqref{assum} implies the solution $u$ is global
and scatters is not at all obvious.

 In the energy-critical setting, the breakthrough was made by Bourgain's monumental work \cite{Bourgain} in which he introduced the induction on energy method.  Based on this method and the space-localized Morawetz inequality, the spherically symmetric energy-critical case was resolved in $d=3,4.$  Subsequently, by using the same strategy and the modified interaction Morawetz estimate, Colliander et al, \cite{Colliander} resolved the nonradial case in $d=3.$  For further discussion about the defocusing energy-critical NLS,  we refer to \cite{Grillakis,KVAP,RM,V1,V2,VIN}. For focusing case see \cite{KenigMerle2,KV3,Dodson5}.

For the mass-critical case,  Conjecture \ref{conj} was primarily proved for spherically-symmetric $L_x^2$ initial data in dimensions $d\geq 2$,  see \cite{KTM,TVZ}.  By introducing long-time Strichartz estimate method, Dodson in \cite{Dodson1,Dodson2,Dodson3} settled the nonradial case. The reader may turn to \cite{KTM,KVX,Dodson4} for focusing setting.

The first work  dealing with  Conjecture \ref{conj} at nonconserved critical regularity is attributed to Kenig-Merle \cite{kenigmerler2} at the case $d=3,s_c=\f{1}{2}$ by making use of  their pioneered concentration-compactness argument along with Lin-Strauss Morawetz inequality. Note that no additional radial assumption is required in \cite{kenigmerler2} due to the fact that  Lin-Strauss Morawetz inequality has a scale of $\f{1}{2}$.  Murphy in \cite{M2} extended the  result to $d\geq 4.$
   
Now we focus on  the case $0<s_c<\f{1}{2}. $ In  \cite{M3}, under the radial assumption, Murphy handled the case $d=3,0<s_c<\f{1}{2}$ by using long-time Strichartz estimate method and frequency-localized Lin-Strauss Morawetz estimate. However,it seems not work in higher dimensions, especially $d\geq 5$.  To be more precise, following the approach in \cite{M3}, one can obtain the corresponding result of four dimensions effortlessly.  To further generalize that to the higher dimensions, however,  is not at all trivial, since it's tricky to establish long-time Strichartz estimate due to the subquadratic property of the nonlinearity. To circumvent the  barrier, we exploit the spherical symmetry condition and adopt the strategy of using weighted Strichartz norms as in \cite{TVZ}. The key observation is that one can formulate the weighted Strichartz norm which scales exactly the same as the   Strichartz norm of the critical regularity. In doing so, we are liberated from  subtle  technicality comes from nonlocal nature of the  fractional derivative thanks to the fact we place the weight and the derivative at the same height in the sense of scaling which can be exemplified by \eqref{refine1}, \eqref{refine2}.  It's worth  mentioning that by adapting the argument in this paper, one may recover the result in \cite{M3}  for  $0<s_c<\f{1}{2}$ in dimension three. We shall clarify this issue at the appropriate point.

For further discussion about Conjecture \ref{conj}, we refer  to \cite{kv2010,LZ,MMZ,DMMZ}.\\

Now we are in a position to state our main results.
 \begin{theorem}\label{maint1}  Let $d\geq 4$, $0<s_c<\f{1}{2}$.
Assume  that $u:I\times \R^d$ is  a  spherically symmetric  maximal-lifespan solution to \eqref{sch0} such that $u\in L_t^\infty \dot{H}_x^{s_c}(I\times \R^d)$.
Then $u$ is global and scatters, with
    \begin{align}
      S_{\mathbb{R}}(u)\leq C(\|u\|_{L_t^\infty\dot{ H}^{s_c}_x})
    \end{align}
    for some function $C:[0,\infty)\rightarrow[0,\infty).$

\end{theorem}

Adapting the argument in \cite{CW}, one can obtain the local-in-time theory which serves as a basis for the proof of Theorem \ref{maint1}.

\begin{theorem}[Local Well-posedness]\label{wellposed}
Let $d$ and $s_c$ be in the Theorem \ref{maint1}, for any~$ u_0 \in\dot{H}^{s_c}(\mathbb{R}^d)$~and ~$t_0 \in\mathbb{R},$~there exists a unique maximal-lifespan solution ~$u:I\times \mathbb{R}^d\rightarrow\mathbb{C}$ to \eqref{sch0} with ~$u(t_0)=u_0$. Furthermore
\begin{enumerate}
\item (Local existence) $I $ is  an open neighborhood of $t_0.$
\item(Blow up ) If $\sup I$ is finite, then $u$ blows up forward in time . If ${\rm inf}I$ is finite, then ~$u$~ blows up backward in time.
    \item (Scattering and wave operators) If sup ~$I=\infty $~and ~$u$~ does not blow up forward in time, then ~$u$~ scatters forward in time. That is, there exists ~$u_{+}\in \dot{H}^{s_c}(\mathbb{R}^d)$~so that
        \begin{equation}\label{scater}
          \lim_{t\rightarrow \infty}\|u(t)-e^{it\Delta}u_{+}\|_{\dot{H}^{s_c}(\mathbb{R}^d)}=0.
        \end{equation}
        Conversely, for any ~$u_+\in \dot{H}^{s_c}(\mathbb{R}^d)$~there exists a unique solution to \eqref{sch0} defined in a neighborhood of~$ t=\infty $~ such that \eqref{scater}holds. The analogous statements hold backward in time.
        \item(Small data scattering) If ~$\|u_0\|_{\dot{H}^{s_c}(\mathbb{R}^d)}$~is sufficiently small, then $u$ is global and scatters, with ~$S_{\mathbb{R}}(u)\lesssim\|u\|_{\dot{H}^{s_c}(\mathbb{R}^d)}^{\f{d+2}{2}p}.$~
            \end{enumerate}
        \end{theorem}
\begin{remark}
 To prove Theorem \ref{wellposed}, one may first assume the initial data belongs to $H^{s_c}_x$ so that the techniques in \cite{CW} applies and then establish Theorem \ref{wellposed} by using the following stability lemma.
\end{remark}

\begin{lemma}\label{purtabation}
  Let $d\geq 4,$  $I$ be a compact interval, and $\tilde{u}: I\times \R^d\rightarrow \mathbb{C}$ be a solution to the equation
\begin{align}\label{purtabation1}
  \begin{split}
  \left\{\begin{array}{cc}
    (i\partial_t+\Delta)\tilde{u}=F(\tilde u)+e\\
    \tilde{u}(0)=\tilde{u}_0\in \dot{H}^{s_c}_x.
    \end{array}\right.
    \end{split}
  \end{align}
  Suppose
  $$\|\tilde{u}\|_{L_t^\infty \dot{H}^{s_c}_x(I\times \R^d)}\leq E\quad \text{and} \quad \|\tilde u\|_{L_{t,x}^{\f{(d+2)p}{2}}(I\times \R^d)}\leq L,$$
  for some $E, L>0.$  There exists $\varepsilon_1(E,L)$ such that if $u_0 \in \dot{H}^{s_c}_x$ and
  \begin{align}
    \|u_0-\tilde{u}_0\|_{\dot{H}^{s_c}_x}+\||\nabla|^{s_c}e\|_{\rm N(I)}\leq \varepsilon,
  \end{align}
  for some small $0<\varepsilon <\varepsilon_1(E,L),$  then there exists a solution $u$ to the equation \eqref{sch0} with the initial data $u_0$  and a constant $0<c(d)$ such that
  \begin{align}\label{smallness}
    \||\nabla|^{s_c}(u-\tilde{u})\|_{\rm S(I)}\leq C(E,L) \varepsilon^c ;\\
    \||\nabla|^{s_c}u\|_{\rm S(I)}\leq C(E,L);
  \end{align}
  where the definition of $\rm S(I)$ and $\rm N(I)$ can be found in the appendix.
\end{lemma}

We present the details of the proof of Lemma \ref{purtabation} in the Appendix.\\

Now we can sketch the proof of Theorem \ref{maint1}.
\subsection{Reduction to a critical solution}
To prove Theorem \ref{maint1}, we argue by contradiction. Due to Theorem \ref{wellposed}, we know small initial data implies the theory of global existence and scattering. If Theorem \ref{maint1} fails, there  exists a counterexample acting as a threshold. As a consequence of its criticality,  such counterexample must concentrate in frequency and physical space at the same time. Further analysis shows that such special solution  possesses  a wealth of weird properties that a solution should not have in general. Finally, we will show that such properties are inconsistent with the structure of the equation \eqref{sch0}.

\begin{definition}
For $A>0$, we define $\mathcal{B}(A)$ as follows
\begin{align*}
\mathcal{B}(A)=\{u_0\in \dot{H}^{s_c}_x, \text{radial}:&u: I \times \R^d \text{ is a maximal-lifespan solution to \eqref{sch0} with  }\\&
u(0) = u_0 \in \dot{H}_x^{s_c} , \text{then}\,  \sup_{t\in I}\|u\|_{\dot{H}^{s_c}_x}\leq A \}.
\end{align*}
  \end{definition}
 \begin{definition}\label{df1}
   We say $\mathcal{SC}(A)$ holds if for each $u_0 \in \mathcal{B}(A)$,  then $I=\R$ and   $S_{I}(u)<\infty$. Similarly, we say $\mathcal{SC}(A, u_0)$ holds if $u_0 \in \mathcal{B}(A)$,  then $I=\R$ and   $S_{I}(u)<\infty$.
 \end{definition}

  In view of \eqref{df1}, to prove Theorem\ref{maint1},  it suffices to show that $\mathcal{SC}(A)$ holds for each $A>0$. Note that Theorem \ref{wellposed} implies $\mathcal{SC}(A)$ holds whenever  $A$ is sufficiently small.  Consequently, if Proposition \ref{maint1} fails, there exists a critical value $A_c$ such that  $\mathcal{SC}(A)$ holds when  $A<A_c$  but fails when $A> A_c$.  
 In particular, using concentration-compactness method, we can obtain the following key proposition.

\begin{proposition}\label{reduc}
  Let $d \geq 4,  0<s_c<\f{1}{2}$, if Proposition \ref{maint1} fails,  there  exists a critical value $A_c$ and a critical element $u_{0,c} \in \mathcal{B}(A_c)$ such that $\mathcal{SC}(A_c, u_{0,c})$  fails.
    Correspondingly, we call $u_{c}:I \times \R^d$  the critical maximal-lifespan  solution to \eqref{sch0} with $u_{c}(0)=u_{0,c}$.
\end{proposition}

 The derivation of Theorem \ref{reduc} by now is standard. One can refer to \cite{KenigMerle1,kv2010,Homer,M1,M2,M3} for more details.

  The critical solution  $u_c$ in Proposition \ref{reduc} enjoys plenty of additional properties, especially among which is its compactness (modulo scaling), see \cite{KenigMerle1,M1}. For brevity, in what follows we abbreviate the critical solution $u_c$ as $u$.
\begin{proposition}
Let $u: I\times \R^d $ be the critical spherically symmetric maximal-lifespan solution to \eqref{sch0}, for each $\eta>0$, there exists functions ~$N:I\rightarrow\mathbb{R}^{+},C:\mathbb{R}^+\rightarrow \mathbb{R}^+$~such that
          \begin{align}\label{almost1}
            \int_{|x|\geq\frac{C(\eta)}{N(t)}}||\nabla|^{s_c}u(t,x)|^2dx+\int_{|\xi|\geq C(\eta)N(t)}|\xi|^{2s_c}|\hat{u}(t,\xi)|^2d\xi<\eta,
          \end{align}
          for all ~$t \in I$.
          We call ~$N(t)$~ the frequency scale function, and ~$C(\eta)$~the compactness modulus function.
        \end{proposition}
\begin{remark}
  \begin{enumerate}
    \item This  definition is adapted to the radial setting. In the  general case, one should also take into account the translation. If we  consider mass-critical case, one more parameter should be added in \eqref{almost1} due to Galilean invariance of \eqref{sch0}.

    \item By the Arzel\`{a}-Ascoli theorem, \eqref{almost1} can be rephrased as
    \begin{align}
      \{u(t):t\in I\}\subset \{\lambda^{\f{d-2s_c}{2}}f(\lambda x): \lambda \in (0,\infty) ~\text{and} ~ f\in K\}
    \end{align}
    where $K $ is a precompact set in $\dot{H}^{s_c}.$  By $\dot{H}^{s_c}\hookrightarrow L^\f{2d}{d-2s_c}_x,$  we know that $u$ is also compact (modulo scaling) in $L_x^{\f{2d}{d-2s_c}},$  that is
    \begin{align}\label{almost2}
      \int_{|x|\geq\frac{C(\eta)}{N(t)}}|u(t,x)|^{\f{2d}{d-2s_c}}dx\leq \eta.
    \end{align}
   \item  We claim that there is a constant $c>0$ such that
    \begin{align}\label{polow}
    \inf_{t \in I}\|u(t)\|_{L_x^{\f{2d}{d-2s_c}}}\geq c.
    \end{align}
     Otherwise,  as $L^{\f{2d}{d-2s_c}}$ norm is left invariant under scaling \eqref{scale},  there exists a sequence $\{N(t_n)^{\f{-d+2s_c}{2}}u\left(t_n,\f{x}{N(t_n)}\right):t_n\in I\}$ such that
     \begin{align}
     N(t_n)^{\f{-d+2s_c}{2}}u\left(t_n,\f{x}{N(t_n)}\right) \rightarrow 0 \quad \text{in}\quad  L^{\f{2d}{d-2s_c}}_x.
     \end{align}
      On the other hand, since $N(t_n)^{\f{-d+2s_c}{2}}u\left(t_n,\f{x}{N(t_n)}\right)$ is also compact in $\dot{H}^{s_c}_x,$  we have
      \begin{align}
      N(t_n)^{\f{-d+2s_c}{2}}u\left(t_n,\f{x}{N(t_n)}\right) \rightarrow 0 \quad \text{in} \quad \dot{H}^{s_c}_x,
      \end{align}
       which contradicts the fact that $u$ blows up.
  \end{enumerate}
  We emphasize that \eqref{polow} has its analogue  in section 6 of \cite{Colliander} which says the potential part  must have lower bound.
  Further, from the compactness property, we may choose $c(\eta)$ sufficiently small such that
  \begin{align}\label{almost3}
    \int_{|x|\leq\frac{c(\eta)}{N(t)}}||\nabla|^{s_c}u(t,x)|^2dx+\int_{|\xi|\leq c(\eta)N(t)}|\xi|^{2s_c}|\hat{u}(t,\xi)|^2d\xi<\eta.
  \end{align}
  \end{remark}

  Next we will record more properties of the critical solution  which will be used in what follows.

\begin{lemma}[Local Constancy\cite{KTM}] \label{locon}
If ~$u:I\times \mathbb{R}^d \rightarrow \mathbb{C}$~is the critical maximal-lifespan  solution to \eqref{sch0}, then there exists ~$\delta=\delta(u)>0$~ so that for all ~$t_0\in I$~
\begin{align}
  [t_0-\delta N(t_0)^{-2},t_0+\delta N(t_0)^{-2}]\subset I.
\end{align}
Moreover,~$N(t)\sim_{u} N(t_0) ~\textrm{for }~|t-t_0|\leq\delta N(t_0)^{-2}.$~
\end{lemma}

Due to Lemma \ref{locon}, we can subdivide the lifespan interval $I$ into several characteristic subintervals $J_k$ such that
\begin{align}\label{Isub}
I=\cup_{k}J_k,\quad N(t)\sim N_k ~\text{when} ~ t \in J_k   ~\text{with}~  |J_k|\sim N_k^{-2}.
\end{align}
 The following result can be directly derived from Lemma \ref{locon}.
\begin{corollary}\label{fscale}
Let ~$u:I\times \mathbb{R}^d\rightarrow \mathbb{C}$~be the critical maximal-lifespan solution to \eqref{sch0}. If ~$T$~ is a finite endpoint of~$ I$, then ~$N(t)\gtrsim_{u} |T-t|^{-1/2}$~. In particular, $\lim_{t\rightarrow T}N(t)=\infty.$
\end{corollary}
Finally we relate the frequency function $N(t)$ to spacetime norm by the following lemma.
\begin{lemma}[Spacetime Bound \cite{KTM}]\label{lowupper}
   Let ~$u:I\times \mathbb{R}^d\rightarrow \mathbb{C}$~ be the  critical maximal-lifespan solution to \eqref{sch0},  for each interval $J\subset I$,  we have
  \begin{align}
    \int_{J}N(t)^2dt\lesssim_{u}\||\nabla|^{s_c}u\|_{L_t^2L_x^{\f{2d}{d-2}}(J\times \mathbb{R}^d)}^2\lesssim_{u}1+\int_JN(t)^2dt.
  \end{align}
\end{lemma}

\begin{remark}
  Owing to \eqref{Isub}, $\int_{I} N(t)^2 dt$ can be rewritten as follows:
\begin{align*}
  \int_{I}N(t)^2 dt=\sum_{k}N_k^2|J_k|\sim \#\{J_k\}
\end{align*}
the above formula indicates that the integral of $\int_{I}N(t)^2 dt$ equals to counting the number of the subintervals $J_k \subset I$.
\end{remark}

By rescaling argument,  we can also ensure
  \begin{align}\label{lowfre1}
  N(t)\leq 1
  \end{align}
  at least on the interval $J$ which is one direction of maximal lifespan  of $u$, say $[0,\sup(I))$. For the sake of exposition, we may harmlessly identify $J$ as $I$. For further discussion, see\cite{TVZ}.

 To prove Theorem \ref{maint1}, it suffices to show that the critical solution  in Theorem \ref{reduc}  does not exist. To this end, the paper is organized  as follows: In Section \ref{notation} we will present some basic tools.  In Section \ref{further}, we  will introduce the weighted Strichartz norm and the associated Strichartz estimate. In Section \ref{frequency}, we will establish frequency-localized Morawetz estimate, as a result, we will show that the weighted Strichartz norm of high frequency portion of the solution $u$ will stay bounded, the fact which we will apply directly to rule out the critical solution.  In Section \ref{verification}, we will show that the frequency scale function $N(t)$ can't go to zero. Together with \eqref{lowfre1}, ultimately we will preclude the critical solution in  Section \ref{rule out}.

\vskip0.15cm

\textbf{Acknowledgments:} This work was supported in part by the National Natural Science Foundation of China under grant No.11671047.

\section{Notation and some basic tools}\label{notation}

We write $X\lesssim Y$  or $Y\gtrsim  X$ whenever $X\leq CY$  for some constant $C>0$ and use $O(Y)$ to denote any quantity $X$ such that  $|X|\lesssim Y.$
 If $X\lesssim Y$  and $Y \lesssim X$  hold simultaneously, we abbreviate that by $X\sim Y.$  Without special clarification, the implicit constant $C$  can vary from line to line. We use Japanese bracket $\langle x\rangle$ to denote $(1+|x|^2)^{\f{1}{2}}.$  We denote by $X\pm$  quantity of the form $X\pm\varepsilon$ for any $\varepsilon>0.$

For any spacetime slab $I\times \R^d,$  we use $L_t^q L_x^r(I\times \R^d)$  to denote the Banach space of functions $u:I\times \R^d \rightarrow \mathbb{C}$  whose norm is
\begin{align*}
  \|u\|_{L_t^q L_x^r(I\times \R^d)}:=\left(\int_I\|u(t)\|_{L_x^r}^qdt\right)^{\f{1}{q}}<\infty,
\end{align*}
with the appropriate modification for the case $q$ or $r$  equals to infinity. When $q=r,$  for brevity, sometimes we write it as $L_{t,x}^q.$  One more thing to be noticed is that without obscurity we will use $L_t^q L_x^r$  with $L^q L^r$  interchangeably.

We define the Fourier transform on $\R^d$ by
\begin{align*}
  \hat{f}:=(2\pi)^{-\f{d}{2}}\int_{\R^d}e^{-ix\xi}f(x)dx,
\end{align*}
and the homogeneous Sobolev norm as
\begin{align*}
  \|f\|_{\dot{H}^s(\R^d)}:=\||\nabla|^s f\|_{L_x^2(\R^d)},
\end{align*}
where
\begin{align*}
  \widehat{|\nabla|^s f}(\xi):=|\xi|^s \hat{f}(\xi).
\end{align*}
Next we will present the Littlewood-Paley decomposition .

Let $\phi(\xi)$  be a radial bump function supported in the ball $\{\xi\in \R^d :|\xi|\leq \f{11}{10}\}$  and equals to $1$  on the ball $\{\xi\in \R^d: |\xi|\leq 1\}.$  For each number $N>0,$  we define
\begin{align*}
    \widehat{P_{\leq N}f}(\xi):=&\varphi\big(\frac{\xi}{N}\big)\hat{f}(\xi), \\
    \widehat{P_{> N}f}(\xi):=&\big(1-\varphi(\frac{\xi}{N})\big)\hat{f}(\xi), \\
    \widehat{P_{N}f}(\xi):=&\big(\varphi(\frac{\xi}{N})-\varphi(\frac{2\xi}{N})\big)\hat{f}(\xi),
\end{align*}
with similar definitions for $P_{< N}$ and $P_{\geq N}$. Moreover,  we define
\begin{align*}
    P_{M<\cdot\leq N}:=P_{\leq N}-P_{\leq M},
\end{align*}
whenever $M<N$. Also there are the following Bernstein inequalities for the  Littlewood-Paley operators:
\begin{equation*}
   \left\{ \aligned
   & \||\nabla|^s P_{\leq N}f\|_{L^q}\lesssim N^s\|P_{\leq N}f\|_{L^q}\lesssim N^s\|f\|_{L^q}, \\
   & \|P_{>N}f\|_{L^q}\lesssim N^{-s}\||\nabla|^s P_{>N}f\|_{L^q}\lesssim N^{-s}\||\nabla|^s f\|_{L^q}, \\
   & \||\nabla|^{\pm s}P_{N}f\|_{L^q}\lesssim N^{\pm s}\|P_N f\|_{L^q}\lesssim N^{\pm s}\|f\|_{L^q}, \\
   &\|P_{\leq N}f\|_{L^q}\lesssim N^{d(\frac{1}{p}-\frac{1}{q})}\|P_{\leq N}f\|_{L^p}, \\
   &\|P_N f\|_{L^q}\lesssim N^{d(\frac{1}{p}-\frac{1}{q})}\|P_{\leq N}f\|_{L^p}.
   \endaligned
  \right.
\end{equation*}
where $1\leq p\leq q\leq \infty$ .

\begin{lemma}[Fractional product rule \cite{CW}]\label{product}
Let ~$s>0$~ and ~$1<r,r_j,q_j<\infty$~ satisfy ~$\frac{1}{r}=\frac{1}{r_j}+\frac{1}{q_j}$~for ~$j=1,2,$~then
\begin{align}
  \||\nabla|^s(fg)\|_{L^r_x}\lesssim\|f\|_{L_x^{r_1}}\||\nabla|^sg\|_{L^{q_1}_x}+\||\nabla|^sf\|_{L_x^{r_2}}\|g\|_{L^{q_2}_x}.
\end{align}
  \end{lemma}
  We will also need the following  chain rule for fractional order derivatives. One can turn to \cite{CW} for more details.
\begin{lemma}[Fractional chain rule]\label{chain}
Suppose ~$G\in C^1(\mathbb{C})$~and ~$s\in (0,1].$~Let ~$1<r<r_2<\infty $~and ~$1<r_1\leq \infty $~be such that~$\frac{1}{r}=\frac{1}{r_1}+\frac{1}{r_2},$~then
\begin{align}\label{frac}
  \||\nabla|^sG(u)\|_{L^r_x}\lesssim\|G'(u)\|_{L_x^{r_1}}\||\nabla|^su\|_{L_x^{r_2}}.
\end{align}
\end{lemma}
When the function $G$ is no longer $\mathbb{C}^1$, but merely H\"older continuous, we have the following chain rule:
\begin{lemma}[Fractional chain rule for H\"older continuous function \cite{V2}]\label{Holderd}
  Let $G$ be a H\"older continuous function of order $0<\alpha<1$. Then for every $0<s<\alpha, 1<p<\infty$, and $\f{s}{\alpha}<\sigma<1$ we have
  \begin{align}
    \||\nabla|^s G(u)\|_p\lesssim \||u|^{\alpha-\f{s}{\sigma}}\|_{p_1}\||\nabla|^\sigma u\|_{\f{s}{\sigma}p_2}^{\f{s}{\sigma}}
  \end{align}
  provided $\f{1}{p}=\f{1}{p_1}+\f{1}{p_2}$  and $(1-\f{s}{\alpha \sigma})p_1>1$.
\end{lemma}
The classical H\"ormander-Mikhlin theorem concerns about the sufficient condition required for a function to be an  $L^p (1<p<\infty)$ multiplier. We should adapt the usual one to be suited for our case and present here the extension form with the power weights. One can refer to \cite{Stein} for further discussion.
  \begin{lemma}\label{H-M}
    Let $T$ be a H\"ormander-Mikhlin multiplier defined on tempered function $f$ i.e,
    \begin{align*}
      \widehat{T f}(\xi):=m(\xi)\hat{f}(\xi),
    \end{align*}
      with its symbol $m(\xi)$ satisfying the following pointwise estimate
    \begin{align*}
      |\nabla^{\alpha} m(\xi)|\lesssim_{\alpha}|\xi|^{-|\alpha|},
    \end{align*}
    for every nonnegative multi-index $\alpha.$ Then for any $1<p<\infty,$ and $-\f{d}{p}<s<d-\f{d}{p},$ we have
    \begin{align}\label{weiBer}
      \||x|^{s}Tf\|_{L_x^p}\lesssim_{p,s}\||x|^{s}f\|_{L_x^p}
    \end{align}
    for all $f$ such that right-hide side is finite.
  \end{lemma}
  \begin{remark}
    In particular, the operator $N^{-s}|\nabla|^{s}P_{<N}$ and $N^s|\nabla|^{-s}P_{\geq N}$ are all H\"ormander-Mikhlin multiplier, as well as the frequency localized operator $P_{N}, P_{\gtrless N}.$
  \end{remark}

At the end of this section,  we will record some fundamental tools.  One  can find details in  \cite{TVZ} and the materials therein .
\begin{lemma}[Hardy-Littlewood-Sobolev Inequality]\label{HLS}
   Let $1< p,q < \infty, d\geq 1, 0<s<d,$ and $\alpha, \beta \in \R$ obey the condition
   \begin{align*}
     \alpha>-\f{d}{p'}\\
     \beta>-\f{d}{q'}\\
     1\leq \f{1}{p}+\f{1}{q}\leq 1+s
   \end{align*}
   and the scaling condition
   \begin{align*}
     \alpha+\beta-d+s=-\f{d}{p'}-\f{d}{q'},
   \end{align*}
  Then for any spherically symmetric $u: \R^d \rightarrow \mathbb{C},$ we have
   \begin{align}\label{raso}
     \||x|^{\beta}u\|_{L^{q'}(\R^d)}\lesssim_{\alpha, \beta, p, q, s}\||x|^{-\alpha}|\nabla|^s u\|_{L^p(\R^d)}.
   \end{align}
 \end{lemma}

\begin{lemma}\label{Jap}
  If $f:\R^d \rightarrow \mathbb{C}, 1<p<\infty, 0<\alpha<\f{d}{p},$ and $N>0,$ then
  \begin{align}
    \||x|^{-\alpha}P_{<N}f\|_{L^p(\R^d)}\lesssim_{\alpha, p}\langle N\rangle^{\alpha}\|\langle x\rangle^{-\alpha}f\|_{L^p(\R^d)}.
  \end{align}
\end{lemma}

\section{ Weighted  Strichartz inequality}\label{further}
Motivated by the work of \cite{TVZ} which handled the mass-critical case, we adapt the argument to tackle the case without conserved quantities. In practice, we  introduce weighted Strichartz norm suited for our case. To be more precise, we define $\|u\|_{\mathcal{S}(I\times \R^d)}$ and $\|u\|_{\mathcal{N}(I\times \R^d)}$ respectively as follows:
\begin{align*}
  \|u\|_{\mathcal{S}(I\times \R^d)}&=\||x|^{-\f{1+\varepsilon}{2}}|\nabla|^{\f{1-\varepsilon}{2}+s_c}u\|_{L_{t,x}^2(I\times \R^d)}+\||\nabla|^{s_c}u\|_{L_t^\infty L_x^2(I\times \R^d)};\\
  \|u\|_{\mathcal{N}(I\times \R^d)}&=\||x|^{\f{1+\varepsilon}{2}}|\nabla|^{-\f{1-\varepsilon}{2}+s_c}u\|_{L_{t,x}^2(I\times \R^d)};
\end{align*}
where $\varepsilon>0$  is a sufficiently small constant depending on $d$ and $s_c$.
By Lemma \ref{H-M}, we obtain that corresponding  Bernstein inequalities with respect to the norms $\|u\|_{\mathcal{S}(I\times \R^d)}$ and  $\|u\|_{\mathcal{N}(I\times \R^d)}$.
\begin{lemma}\label{weiber}
For any $s> 0$ and dyadic number $N>0$, we have
  \begin{align}
    \||\nabla|^{s}u_{<N}\|_{\mathcal{S}(I\times \R^d)}&\lesssim  N^s \|u_{<N}\|_{\mathcal{S}(I\times \R^d)};\\
  \||\nabla|^{-s}u_{>N}\|_{\mathcal{S}(I\times \R^d)}&\lesssim N^{-s}\|u_{>N}\|_{\mathcal{S}(I\times \R^d)};\\
  \||\nabla|^{s}u_{<N}\|_{\mathcal{N}(I\times \R^d)}&\lesssim N^{s}\|u_{<N}\|_{\mathcal{N}(I\times \R^d)};\\
  \||\nabla|^{-s}u_{>N}\|_{\mathcal{N}(I\times \R^d)}&\lesssim N^{-s}\|u_{>N}\|_{\mathcal{N}(I\times \R^d)}.
  \end{align}
\end{lemma}
The association of $\|u\|_{\mathcal{S}(I\times \R^d)}$  and  $\|u\|_{\mathcal{N}(I\times \R^d)}$ with equation \eqref{sch0} is illuminated by the following weighted Strichartz estimate and radial Sobolev embedding.
\begin{proposition}[Weighted Strichartz estimate \cite{Vilela}]
Let $u,G:I\times \R^d \rightarrow \mathbb{C}$ satisfy $(i\partial_t+\Delta)u=G$ in the sense of distributions, then we have
\begin{align}\label{weightstr}
\|u\|_{\mathcal{S}(I\times \R^d)}\lesssim \|u(t_0)\|_{\dot{H}^{s_c}(\R^d)}+\|G\|_{\mathcal{N}(I\times \R^d)},
\end{align}
for all $t_0 \in I$.
\end{proposition}

Using  \eqref{raso}, we will get the following radial Sobolev embedding.
\begin{lemma}[Radial Sobolev embedding]
Let $u$ be spherically symmetric and $d\geq 4$, then we have
\begin{align}\label{radso}
  \||\nabla|^{s_c}u\|_{L_t^2 L_x^{\f{2d}{d-2}}}&\lesssim \|u\|_{\mathcal{S}}.
  \end{align}
 \end{lemma}

\begin{lemma}
  If $u, v: I\times \R^d \rightarrow \mathbb{C}$ are spherically symmetric and $d\geq 4$,  then
  \begin{align}\label{nonesti}
    \||u|^{\f{4}{d-2s_c}}v\|_{\mathcal{N}} \lesssim \||\nabla|^{s_c}u\|_{L_t^\infty L_x^2}^{\f{4}{d-2s_c}}\||\nabla|^{s_c}v\|_{L_t^2 L_x^{\f{2d}{d-2}}}\lesssim \||\nabla|^{s_c}u\|_{L_t^\infty L_x^2}^{\f{4}{d-2s_c}} \|v\|_{\mathcal{S}}.
    \end{align}
    \end{lemma}
\begin{proof}
{\bf Case I $d=4$, $p=\f{2}{2-s_c}>1$}

 By \eqref{raso}, we obtain
 \begin{align}\label{nonla}
   \||x|^{\f{1+\varepsilon}{2}}u\|_{L_{t,x}^2}\lesssim \||\nabla|^{\f{1-\varepsilon}{2}}u\|_{L_t^2 L_x^{\f{4}{3}}}.
 \end{align}
 By the definition of $\mathcal{N}$, \eqref{nonla} implies
 \begin{align}\label{nonlb}
   \||u|^{\f{2}{2-s_c}}v\|_{\mathcal{N}}&\lesssim\||\nabla|^{s_c}(|u|^{\f{2}{2-s_c}}v)\|_{L_t^2 L_x^{\f{4}{3}}}.
 \end{align}
 Continuing from \eqref{nonlb}, by Lemma\ref{product},  Lemma \ref{chain} and \eqref{radso} we have
 \begin{align*}
   \text{RHS of} \eqref{nonlb} &\lesssim \||\nabla|^{s_c}|u|^{\f{2}{2-s_c}}\|_{L_t^\infty L_x^{\f{4}{2+s_c}}}\|v\|_{L_t^2 L_x^{\f{4}{1-s_c}}}+\||u|^{\f{2}{2-s_c}}\|_{L_t^\infty L_x^2}\||\nabla|^{s_c}v\|_{L_t^2 L_x^4}\\
    &\lesssim \||\nabla|^{s_c}u\|_{L_t^\infty L_x^2}^{\f{2}{2-s_c}}\||\nabla|^{s_c}v\|_{L_t^2 L_x^4}\\
    &\lesssim \||\nabla|^{s_c}u\|_{L_t^\infty L_x^2}^{\f{2}{2-s_c}}\|v\|_{\mathcal{S}}.
    \end{align*}

{\bf Case II: $d\geq 5$, $p=\f{4}{d-2s_c}<1$}

 If $0<s_c<\f{d-2}{2(d-1)}$,  by the definition of $\mathcal{N}$, \eqref{raso} implies
 \begin{align}\label{nonlbb}
  \||u|^{\f{4}{d-2s_c}}v\|_{\mathcal{N}}&\lesssim\||u|^{\f{4}{d-2s_c}}v\|_{L_t^2 L_x^{\f{2d}{d+2-2s_c}}}.
 \end{align}
 Continuing from \eqref{nonlbb}, by the H\"older inequality and \eqref{radso} we have
  \begin{align*}
  &\lesssim \|u\|_{L_t^\infty L_x^{\f{2d}{d-2s_c}}}^{\f{4}{d-2s_c}}\|v\|_{L_t^2 L_x^{\f{2d}{d-2(1+s_c)}}}\\
    &\lesssim \||\nabla|^{s_c}u\|_{L_t^\infty L_x^2}^{\f{4}{d-2s_c}}\|v\|_{\mathcal{S}}.
    \end{align*}
    When $\f{d-2}{2(d-1)}\leq s_c<\f{1}{2}$,  denoting $\bar{s}=(s_c-\f{d-2}{2(d-1)})+$, similarly by \eqref{raso} we have
    \begin{align}\label{nonlac}
      \||u|^{\f{4}{d-2s_c}}v\|_{\mathcal{N}}&\lesssim \||\nabla|^{\bar{s}}(|u|^{\f{4}{d-2s_c}}v)\|_{L_t^2 L_x^{\f{2d}{d+2-2(s_c-\bar{s})}}}.
    \end{align}
    Continuing from \eqref{nonlac}, by Lemma \ref{product} we have
    \begin{align*}
     &\quad\text{RHS of }\eqref{nonlac}\\
      &\lesssim \||\nabla|^{\bar s}|u|^{\f{4}{d-2s_c}}\|_{L_t^\infty L_x^{\f{d}{2+\bar s}}}  \|v\|_{L_t^2 L_x^{\f{2d}{d-2-2s_c}}}+\||u|^{\f{4}{d-2s_c}}\|_{L_t^\infty L_x^{\f{d}{2}}}  \||\nabla|^{\bar{s}}v\|_{L_t^2 L_x^{\f{2d}{d-2-2(s_c-\bar s)}}}\\
      &\lesssim \||\nabla|^{\bar s}|u|^{\f{4}{d-2s_c}}\|_{L_t^\infty L_x^{\f{d}{2+\bar s}}}  \||\nabla|^{s_c}v\|_{L_t^2 L_x^{\f{2d}{d-2}}}+\||\nabla|^{s_c}u\|_{L_t^\infty L_x^2}^{\f{4}{d-2s_c}}\||\nabla|^{s_c}v\|_{L_t^2 L_x^{\f{2d}{d-2}}}.
      \end{align*}
 To complete the proof, it suffices to show that
 \begin{align}\label{Hode}
    \||\nabla|^{\bar s}|u|^{\f{4}{d-2s_c}}\|_{L_t^\infty L_x^{\f{d}{2+\bar s}}}\lesssim \||\nabla|^{s_c}u\|_{L_t^\infty L_x^2}^{\f{4}{d-2s_c}}.
 \end{align}
  To this end, setting $\sigma=\f{\bar s}{p}+\tilde \varepsilon$,  where $\tilde \varepsilon$ is a sufficiently small positive constant(say,  $\tilde \varepsilon=\f{1}{2^{10}}$ ). Using Lemma \ref{Holderd} with $\alpha$ being replaced by $\f{4}{d-2s_c}$,  we have
\begin{align}
  \||\nabla|^{\bar s}|u|^{\f{4}{d-2s_c}}\|_{L_t^\infty L_x^{\f{d}{2+\bar s}}} \lesssim \||u|^{p-\f{\bar{s}}{\sigma}}\|_{L_t^\infty L_x^{\f{2d\sigma}{(d-2s_c)p\tilde \varepsilon}}}\||\nabla|^{\sigma}u\|_{L_t^\infty L_x^{\bar p}}^{\f{\bar s}{\sigma}},
\end{align}
  where $\bar p=\f{2d\bar s}{2(2+\bar s)\sigma-(d-2s_c)p\tilde\varepsilon}$, using Sobolev inequality we have \eqref{Hode}.
\end{proof}

By the local well-posed theory, for example see \cite{Ca}, one has
\begin{align}\label{localstr}
  \||\nabla|^{s_c}u\|_{L_t^2 L_x^{\f{2d}{d-2}}(J\times \R^d)}<\infty
\end{align}
for any compact interval $J$  contained in the the maximal lifespan interval  $I$. As a direct application of \eqref{nonesti}, we obtain the following result which, in some sense,  can be viewed as an extension of  \eqref{localstr} in the weighted norm.
\begin{corollary}\label{localfin}
  Let $u:I\times \R^d \rightarrow \mathbb{C}$ be a  spherically symmetric maximal-lifespan  solution to \eqref{sch0} then
  \begin{align*}
    \|u\|_{\mathcal{S}(J\times \R^d)}<\infty, \quad \text{for all compact set}\quad J\subset\subset I.
  \end{align*}
\end{corollary}
\begin{proof}
  Using \eqref{assum}, \eqref{weightstr} \eqref{nonesti}and \eqref{localstr}, we obtain
  \begin{align*}
    \|u\|_{\mathcal{S}(J\times \R^d)}&\lesssim  1+\||u|^{\f{4}{d-2s_c}}u\|_{\mathcal{N}}\\
    &\lesssim \||\nabla|^{s_c}u\|_{L_t^\infty L_x^2}^{\f{4}{d-2s_c}}\||\nabla|^{s_c}u\|_{L_t^2 L_x^{\f{2d}{d-2}}}\\
    &<\infty.
  \end{align*}
\end{proof}
Next, we will give some refined nonlinear estimates which will be used to control the nonlinear interaction.
\begin{proposition}[Refined nonlinear estimate]
Let $u,v: I\times \R^d \rightarrow \mathbb{C}$ be spherically symmetric, then we have
\begin{align}\label{refine1}
  \||\nabla|^{\f{1-\varepsilon}{2}-s_c}O(|u|^{\f{4}{d-2s_c}}v)\|_{\mathcal{N}} \lesssim \||\nabla|^{(\f{d}{2}-\f{(d-2s_c)(1+s_c+\varepsilon)}{4})-}u\|_{L_t^\infty L_x^2}^{\f{4}{d-2s_c}}\||\nabla|^{(-\f{1-\varepsilon}{2})+}v\|_{\mathcal{S}},
  \end{align}
  \begin{align}\label{refine2}
     \||\nabla|^{\f{1-\varepsilon}{2}-s_c}O(|u|^{\f{4}{d-2s_c}}v)\|_{\mathcal{N}} \lesssim \||\nabla|^{\f{1+\varepsilon}{2}} u\|_{L_t^\infty L_x^2}^{\f{4}{d-2s_c}}\||\nabla|^{\f{1-\varepsilon}{2}-s_c+p(s_c-\f{1+\varepsilon}{2})}v\|_{\mathcal{S}}.
  \end{align}
  \begin{proof}
  By the definition of $\mathcal{N}$ and the H\"older inequality and Lemma \ref{HLS}, we estimate \eqref{refine1} as
  \begin{align*}
   &\quad \||\nabla|^{\f{1-\varepsilon}{2}-s_c}O(|u|^{\f{4}{d-2s_c}}v)\|_{\mathcal{N}} \\
    &\lesssim \||x|^{\f{1+\varepsilon}{2}}O(|u|^{\f{4}{d-2s_c}}v)\|_{L_{t,x}^2}\\
    &\lesssim \||x|^{1+\varepsilon+s_c}|u|^{\f{4}{d-2s_c}}\|_{L_t^\infty L_x^{\infty-}}\||x|^{-\f{1+\varepsilon}{2}-s_c}v\|_{L_t^2 L_x^{2+}}\\
    &\lesssim \||\nabla|^{(\f{d}{2}-\f{(d-2s_c)(1+s_c+\varepsilon)}{4})-}u\|_{L_t^\infty L_x^2}^{\f{4}{d-2s_c}}\||\nabla|^{(-\f{1-\varepsilon}{2})+}v\|_{\mathcal{S}}.
    \end{align*}
     Similarly for \eqref{refine2}, we have
    \begin{align*}
     &\quad  \||\nabla|^{\f{1-\varepsilon}{2}-s_c}O(|u|^{\f{4}{d-2s_c}}v)\|_{\mathcal{N}} \\
    &\lesssim \||x|^{\f{1+\varepsilon}{2}}O(|u|^{\f{4}{d-2s_c}}v)\|_{L_{t,x}^2}\\
    &\lesssim \|u\|_{L_t^\infty L_x^{\f{2d}{(d-1-\varepsilon))}}}^{\f{4}{d-2s_c}}\||x|^{\f{1+\varepsilon}{2}}v\|_{L_t^2 L_x^{\f{2d}{d-(d-1-\varepsilon)p}}}\\
    &\lesssim  \||\nabla|^{\f{1+\varepsilon}{2}} u\|_{L_t^\infty L_x^2}^{\f{4}{d-2s_c}}\||\nabla|^{\f{1-\varepsilon}{2}-s_c+p(s_c-\f{1+\varepsilon}{2})}v\|_{\mathcal{S}}.
    \end{align*}
  \end{proof}
 \end{proposition}
 \begin{remark}
    \eqref{refine1}  is very useful  when $u$ is low frequency and $v$ is high frequency, as it transfers plenty of derivatives from high frequency to low frequency via the appropriate distribution of weight.
 \end{remark}
\section{Frequency-localized  Morawetz estimate}\label{frequency}
In this part we will primarily establish the following frequency-localized Morawetz inequality.
\begin{proposition}[Frequency-localized Morawetz estimate]\label{Morawetz1}
    Let $d \geq 4$ and  $u:I\times \mathbb{R}^d  \rightarrow \mathbb{C}$ be the critical spherically symmetric maximal-lifespan solution to \eqref{sch0} which  obeys \eqref{assum}, \eqref{lowfre1}, then we have
    \begin{align}
      \lim_{N\rightarrow \infty}N^{2s_c}\int_{I}\int_{\mathbb{R}^d}\f{|\nabla u_{<N}(t,x)|^2}{|Nx|^{1+\varepsilon}}dxdt=0.
    \end{align}
    \end{proposition}

To prove Proposition \ref{Morawetz1}, we will  first exploit some nontrivial facts about the critical solution $u$.
\begin{lemma} \label{decayhi}
  Let $u:I\times \R^d \rightarrow \mathbb{C}$ be the critical spherically symmetric  maximal-lifespan solution to \eqref{sch0} which  obeys \eqref{assum}, \eqref{lowfre1}. Then for each $\theta>0$, we have
  \begin{align}
    \lim_{N \rightarrow \infty}\left(\||\nabla|^{s_c}u_{\geq N}\|_{L_t^\infty L_x^2(I\times \R^d)}+\f{1}{N^{\theta}}\||\nabla|^{\theta+s_c}u_{<N}\|_{L_t^\infty L_x^2(I\times \R^d)}\right)=0.
  \end{align}
\end{lemma}
\begin{proof}
  By \eqref{almost1} and \eqref{lowfre1}, we have that
  \begin{align*}
    \lim_{N\rightarrow \infty}\||\nabla|^{s_c}u_{>N}\|_{L_t^\infty L_x^2}=0.
  \end{align*}
  Now we turn to  proving the second term, we split $u_{<N}$ as $u_{<N}:=u_{<\sqrt{N}}+u_{\sqrt{N}\leq.<N}$  then by Bernstein inequality, we have
  \begin{align*}
    &\f{1}{N^{\theta}}\||\nabla|^{\theta+s_c}u_{<N}\|_{L_t^\infty L_x^2}\\
    \lesssim &  \f{1}{N^{\theta}}\||\nabla|^{\theta+s_c}u_{<\sqrt{N}}\|_{L_t^\infty L_x^2}+\f{1}{N^{\theta}}\||\nabla|^{\theta+s_c}u_{\sqrt{N}\leq .< N}\|_{L_t^\infty L_x^2}\\
    \lesssim &\f{1}{N^{\f{\theta}{2}}}\||\nabla|^{s_c}u_{<\sqrt{N}}\|_{L_t^\infty L_x^2}+\||\nabla|^{s_c}u_{\sqrt{N}\leq .< N}\|_{L_t^\infty L_x^2}\\
    &\rightarrow 0 \quad \text {as} \quad N\rightarrow \infty.
  \end{align*}
  \end{proof}
In view of this Lemma \ref{decayhi}, we can reformulate Proposition \ref{Morawetz1} as follows
  \begin{theorem}[Frequency-localized Morawetz estimate $\uppercase\expandafter{\romannumeral 1}$]
    Let $d\geq 4,0<\eta<1, $ and $u: I\times \R^d$ be the critical spherically symmetric  maximal-lifespan solution to \eqref{sch0} which satisfies  \eqref{assum},\eqref{lowfre1}.  Then there exits $\delta>0$ with the following property: given any $N>0$ such that
    \begin{align}\label{Nassum}
      \||\nabla|^{s_c}u_{\geq N}\|_{L_t^\infty L_x^2(I\times \R^d)}+\f{1}{N^{\theta}}\||\nabla|^{\theta+s_c} u_{<N}\|_{L_t^\infty L_x^2(I\times \mathbb{R}^d)}\leq \delta,
    \end{align}
    we have
    \begin{align}\label{maro1a}
     N^{2s_c}\int_{I}\int_{\R^d}\f{|\nabla u_{<N}(t,x)|^2}{|Nx|^{1+\varepsilon}}dxdt \leq \eta.
    \end{align}
  \end{theorem}
By scaling invariance of the equation \eqref{sch0}, we may choose $N=1.$  By a limiting argument, we may then take $I$ to be compact. Indeed, observe that by
Corollary \ref{localfin}, the left-hand side of \eqref{maro1a} varies continuously on $I$ and goes to zero when $I$ shrinks to a point. Thus, by standard continuity argument, it suffices to show  the following bootstrap version of Proposition \ref{Morawetz1}.

\begin{proposition}[Frequency-localized Morawetz estimate $\uppercase\expandafter{\romannumeral 2}$]\label{maro2}
   Let $d\geq 4,0<\eta<1$, and $u: I\times \R^d $ be the critical symmetric solution to \eqref{sch0} which  satisfies  \eqref{assum},\eqref{lowfre1}.  Then there exits $\delta>0$ with the following property:
     \begin{align}\label{maro3}
      \||\nabla|^{s_c}u_{hi}\|_{L_t^\infty L_x^2(I\times \R^d)}+\||\nabla|^{\theta+s_c} u_{lo}\|_{L_t^\infty L_x^2(I\times \mathbb{R}^d)}\leq \delta,
    \end{align}
    where $u_{hi}:=u_{\geq 1}$  and $u_{lo} :=u_{<1},$ such that we also have bootstrap hypothesis: if
    \begin{align}\label{maro1}
      Q_I:=\int_{I}\int_{\R^d}\f{|\nabla u_{lo}(t,x)|^2}{|x|^{1+\varepsilon}}dxdt\leq 2\eta ,
    \end{align}
    then we have
     \[Q_I\leq \eta.\]
\end{proposition}

  In order to prove Proposition \ref{maro2}, we will primarily establish the corresponding estimate for low and high frequency portion of the solution $u$.
\begin{lemma}[Low and high frequency bound]
Under the conditions  of Proposition \ref{maro2}, we have the following estimates:
 \begin{align}
   \||\nabla|^{\f{1+\varepsilon}{2}} u_{lo}\|_{\mathcal{S}(I\times \R^d)}&\lesssim \eta^{1/2},\label{maro21}\\
   \|\nabla u_{lo}\|_{L_t^2 L_x^{\f{2d}{d-2(1-\varepsilon_0)}}(I\times \R^d)}&\lesssim \eta^{1/2},\label{maro23}\\
   \|u_{hi}\|_{\mathcal{S}(I\times \R^d)}&\lesssim \delta +\delta^{\f{4}{d-2s_c}}.\label{highes}
    \end{align}
    where $\varepsilon_0(d)>0$ is sufficiently small.
  \end{lemma}
  \begin{proof}
   From the definition of $\mathcal{S}$, Lemma \ref{weiber} \eqref{maro3} and \eqref{maro1} we derive \eqref{maro21} by choosing $\delta$ sufficiently small. \eqref{maro23} comes from  \eqref{maro21} and \eqref{raso}.
  Indeed, by Lemma \ref{weiber} and choosing $\varepsilon_0$ sufficiently small, we have
  \begin{align*}
    \|\nabla u_{lo}\|_{L_t^2 L_x^{\f{2d}{d-2(1-\varepsilon_0)}}}\lesssim \||\nabla|^{\f{1+\varepsilon}{2}+s_c+\varepsilon_0}u_{lo}\|_{L_t^2 L_x^{\f{2d}{d-2(1-\varepsilon_0)}}}.
  \end{align*}
   By\eqref{maro21} and \eqref{raso}, we  get \eqref{maro23}.

  Now it suffices to prove \eqref{highes}.
  We denote $P_{hi}:=P_{\geq 1}.$  Obviously
\begin{align*}
  (i\partial_t+\Delta)P_{hi}u=P_{hi}F(u).
\end{align*}
 By Strichartz estimate \eqref{weightstr}, \eqref{maro3} and  splitting  $P_{hi}F(u)$ into
 $$P_{hi}F(u)=P_{hi}F(u_{lo})+P_{hi}(F(u)-F(u_{lo})),$$  we have
  \begin{align}
    \quad   \|u_{hi}\|_{\mathcal{S}} &\lesssim \delta +\|P_{hi}F(u)\|_{\mathcal{N}}\\
    \begin{split}
   \lesssim \delta +\|P_{hi}(|u_{lo}|^{\f{4}{d-2s_c}}|\nabla u_{lo}|)\|_{\mathcal{N}}&+
   \|P_{hi}(|u_{lo}|^{\f{4}{d-2s_c}}|u_{hi}|)\|_{\mathcal{N}}\\
    +\||u_{hi}|^{\f{4}{d-2s_c}}|u_{hi}|\|_{\mathcal{N}}. \label{higha}
    \end{split}
  \end{align}
  For the fourth term of \eqref{higha}, from Proposition \ref{nonesti} and \eqref{maro3} we have
  \begin{align*}
    \||u_{hi}|^{\f{4}{d-2s_c}}|u_{hi}|\|_{\mathcal{N}}&\lesssim \||\nabla|^{s_c}u_{hi}\|_{L^\infty_tL_x^2}^{\f{4}{d-2s_c}}\|u_{hi}\|_{\mathcal{S}}\\
     &\lesssim \delta^{\f{4}{d-2s_c}}\|u_{hi}\|_{\mathcal{S}}.
  \end{align*}
  For the third term of \eqref{higha}, by Lemma \ref{weiber}, \eqref{refine1} and \eqref{maro3}, we have
  \begin{align*}
    &\quad \|P_{hi}(|u_{lo}|^{\f{4}{d-2s_c}}|u_{hi}|)\|_{\mathcal{N}}\\
    &\lesssim \||\nabla|^{\f{1-\varepsilon}{2}-s_c}(|u_{lo}|^{\f{4}{d-2s_c}}|u_{hi}|)\|_{\mathcal{N}}\\
    &\lesssim \||\nabla|^{(\f{d}{2}-\f{(d-2s_c)(1+s_c+\varepsilon)}{4})-}u_{lo}\|_{L_t^\infty L_x^2}^{\f{4}{d-2s_c}}\||\nabla|^{(-\f{1-\varepsilon}{2})+}u_{hi}\|_{\mathcal{S}}\\
    &\lesssim \delta^{\f{4}{d-2s_c}}\|u_{hi}\|_{\mathcal{S}}.
  \end{align*}
For the remained term of \eqref{higha},  by Lemma \ref{weiber}, \eqref{refine1}, \eqref{maro3} and \eqref{maro21} we have
  \begin{align*}
    &\quad \|P_{hi}(|u_{lo}|^{\f{4}{d-2s_c}}|\nabla u_{lo}|)\|_{\mathcal{N}}\\
    &\lesssim \||x|^{\f{1+\varepsilon}{2}}(|u_{lo}|^{\f{4}{d-2s_c}}|\nabla u_{lo}|)\|_{L_{t,x}^2}\\
    &\lesssim \||\nabla|^{(-\f{1-\varepsilon}{2})+}\nabla u_{lo}\|_{\mathcal{S}}\||\nabla|^{(\f{d}{2}-\f{(d-2s_c)(1+s_c+\varepsilon)}{4})-}u_{lo}\|_{L_t^\infty L_x^2}^{\f{4}{d-2s_c}}\\
    &\lesssim \eta^{\f{1}{2}}\delta^{\f{4}{d-2s_c}}.
  \end{align*}
  Putting all these together, we obtain
  \begin{align}
    \|u_{hi}\|_{\mathcal{S}}\lesssim(\delta+\delta^{\f{4}{d-2s_c}})(1+\|u_{hi}\|_{\mathcal{S}}),
  \end{align}
 by Corollary \ref{localfin}, we know $\|u_{hi}\|_{\mathcal{S}}<\infty$, after reorganizing  the term,  we finally derive that
  \begin{align*}
    \|u_{hi}\|_{\mathcal{S}(I\times \R^d)}\lesssim \delta +\delta^{\f{4}{d-2s_c}}.
  \end{align*}

\end{proof}

With the above preparation, we are now ready to prove Proposition  \ref{maro2}. First we need  the following particular form of Morawetz inequality which can be found in \cite{TVZ}.
\begin{lemma}[Morawetz inequality]
   Let $J$ be an interval, let $d\geq 3 $ and let $\phi,G:J\times \R^d \rightarrow \mathbb{C}$ solve the equation
  \begin{align*}
    i\partial_t \phi+\Delta\phi=F(\phi)+G
  \end{align*}
   Let $\varepsilon>0.$ If $\varepsilon$ is sufficiently small depending on $d$,  then we have \begin{align}
    \int_{J}\int_{\mathbb{R}^d}&\left(\f{|\phi(t,x)|^2}{\langle x\rangle^{3+\varepsilon}}+\f{|\phi|^{p+2}}{\langle x\rangle}+\f{|\nabla \phi(t,x)|^2}{\langle x\rangle^{1+\varepsilon}}\right)dxdt\\
    &\lesssim_{\varepsilon} \sup_{t\in J}\||\nabla|^{\f{1}{2}}\phi(t,x)\|_{L_x^2}^2\\
    &+\int_J\int_{\mathbb{R}^d}G(t,x)||\nabla \phi(t,x)|dxdt\\
    &+\int_J \int_{\R^d}\f{1}{\langle x\rangle}|G(t,x)||\phi(t,x)|dxt
    \end{align}
\end{lemma}
  {\bf Proof of Proposition \ref{maro2}. }
Let $P_{lo}:=P_{<1},$ we substitute $\phi$ with $\phi=u_{lo}$, then the corresponding $G$ equals
\begin{align*}
  G:=P_{lo}F(u)-F(P_{lo}u).
\end{align*}
Using Bernstein inequality and \eqref{maro3}, we conclude that
\begin{align}\label{maro3a}
  \int_{I}\int_{\R^d}\f{|\nabla u_{lo}(t,x)|^2}{\langle x\rangle^{1+\varepsilon}}dxdt\lesssim _{\varepsilon}\delta+
  \int_{I}\int_{R^d}|G(t,x)|\left(|\nabla u_{lo}(t,x)|+\f{|u_{lo}(t,x)|}{\langle x\rangle}\right)dxdt.
\end{align}
Note that by Lemma \ref{Jap}
\begin{align*}
\int_{I}\int_{\R^d}\f{|\nabla u_{lo}(t,x)|^2}{|x|^{1+\varepsilon}}dxdt\lesssim \int_{I}\int_{\R^d}\f{|\nabla u_{lo}(t,x)|^2}{\langle x\rangle^{1+\varepsilon}}dxdt,
\end{align*}
it suffices to estimate
\begin{align*}
  \int_I \int_{R^d}|G(t,x)|\left(|\nabla u_{lo}(t,x)|+\f{|u_{lo}(t,x)|}{\langle x\rangle}\right)dxdt \lesssim \delta^{c},
\end{align*}
where $c$ is a given constant to be chosen later.
By the H\"older inequality and \eqref{maro23}, we estimate
\begin{align*}
  &\quad \int_I \int_{R^d}|G(t,x)||\nabla u_{lo}(t,x)|dxdt\\
  &\lesssim  \|G\|_{L_t^2 L_x^{\f{2d}{d+2(1-\varepsilon_0)}}}\|\nabla u_{lo}\|_{L_t^2 L_x^{\f{2d}{d-2(1-\varepsilon_0)}}}\\
  &\lesssim \eta^{\f{1}{2}}\|G\|_{L_t^2 L_x^\f{2d}{d+2(1-\varepsilon_0)}}.
\end{align*}
In dimension $d\geq 4,$ by the H\"older inequality, \eqref{raso} and \eqref{maro23} we have
\begin{align}\label{Hardyu}
\begin{split}
  \int_I\int_{\R^d}|G(t,x)|\f{|u_{lo}(t,x)|}{\langle x\rangle}dxdt &\lesssim \|G\|_{L_t^2 L_x^{\f{2d}{d+2(1-\varepsilon_0)}}}\left\|\f{u_{lo}}{\langle x\rangle}\right\|_{L_t^2 L_x^{\f{2d}{d-2(1-\varepsilon_0)}}}\\
  &\lesssim \eta^{\f{1}{2}}\|G\|_{L_t^2 L_x^{\f{2d}{d+2(1-\varepsilon_0)}}}.
  \end{split}
\end{align}
Thus it is reduced to show
\begin{align}\label{esG}
  \|G\|_{L_t^2 L_x^{\f{2d}{d+2(1-\varepsilon_0)}}}\lesssim_{\eta} \delta^{c}.
\end{align}

 We split $G$ into
 \begin{align*}
   G:=P_{lo}[F(u)-F(u_{lo})]-P_{hi}(F(u_{lo})).
    \end{align*}
  We can  show \eqref{esG} via
   \begin{align}
   \|P_{lo}O(|u_{hi}|&|u_{lo}|^{\f{4}{d-2s_c}}+|u_{hi}|^{1+\f{4}{d-2s_c}})\|_{L_t^2 L_x^{\f{2d}{d+2(1-\varepsilon_0)}}}
     +\|P_{hi}F(u_{lo})\|_{L_t^2 L_x^{\f{2d}{d+2(1-\varepsilon_0)}}}\notag\\
    &\lesssim \|P_{lo}O(|u_{hi}||u_{lo}|^{\f{4}{d-2s_c}}+|u_{hi}|^{1+\f{4}{d-2s_c}})\|_{L_t^2 L_x^{\f{2d}{d+2(1-\varepsilon_0)}}}\label{esG1}\\
      &+\|\nabla P_{hi}F(u_{lo})\|_{L_t^2 L_x^{\f{2d}{d+2(1-\varepsilon_0)}}}.\label{esG2}
   \end{align}
For \eqref{esG1}, by \eqref{assum}, Sobolev embedding, Lemma \ref{weiber}, \eqref{highes}, Bernstein, we estimate as
\begin{align*}
 \eqref{esG1}&\lesssim  \|O(|u_{hi}||u_{lo}|^{\f{4}{d-2s_c}}+|u_{hi}|^{1+\f{4}{d-2s_c}})\|_{L_t^2 L_x^{\f{2d}{d+2(1-\varepsilon_0)}}}\lesssim \||\nabla|^{s_c}u\|_{L_t^\infty L_x^2}^{\f{4}{d-2s_c}}\|u_{hi}\|_{L_t^2 L_x^{\f{2d}{d-2}}}\\
 &\lesssim \delta +\delta^{\f{4}{d-2s_c}}.
\end{align*}

Hence, it is remained to prove
\begin{align*}
  \||\nabla u_{lo}||u_{lo}|^{\f{4}{d-2s_c}}\|_{L_t^2L_x^\f{2d}{d+2(1-\varepsilon_0)}}\lesssim \delta^{c}.
\end{align*}
From \eqref{maro1} we have
\begin{align}
  \||x|^{-\f{1+\varepsilon}{2}}\nabla u_{lo}\|_{L_{t,x}^2}\lesssim \eta^{\f{1}{2}},
\end{align}
and by radial Sobolev embedding \eqref{raso}
\begin{align}
  \||\nabla|^{s}\nabla u_{lo}\|_{L_t^2 L_x^q}\lesssim \||x|^{-\f{1+\varepsilon}{2}}\nabla u_{lo}\|_{L_{t,x}^2}\lesssim \eta^{\f{1}{2}},
\end{align}
for some $q=(\f{2(d-1)}{d-2-\varepsilon})+,s=(\f{d}{q}-\f{d}{2}+\f{1+\varepsilon}{2})-.$ By Bernstein we conclude that
\begin{align}
  \|\nabla u_{lo}\|_{L_t^2 L_x^q}\lesssim \eta^{\f{1}{2}}.
\end{align}
By the H\"older inequality, we get
\begin{align*}
  \||\nabla u_{lo}||u_{lo}|^{\f{4}{d-2s_c}}\|_{L_t^2 L_x^{\f{2d}{d+2(1-\varepsilon_0)}}}\lesssim   \|\nabla u_{lo}\|_{L_t^2 L_x^q} \|u_{lo}\|_{L_t^\infty L_x^r}^{\f{4}{d-2s_c}},
\end{align*}
for some $r=\f{2d(d-1)p}{2(1-\varepsilon_0)(d-1)+d(1-\varepsilon)}->\f{2d}{d-2s_c}$.  By \eqref{maro3}, we have
\[\|u_{lo}\|_{L_t^\infty L_x^r}\lesssim \delta.\]

Combining  the estimate for \eqref{esG1} and \eqref{esG2} we have
\begin{align}
   \|G\|_{L_t^2 L_x^{\f{2d}{d+2(1-\varepsilon_0)}}}\lesssim_{\eta} \delta+\delta^{\f{4}{d-2s_c}}.
\end{align}
Now we can choose $c=\min\{1,\f{4}{d-2s_c}\}$ and  $\delta(\eta)$ sufficiently small, then we complete the proof.
\begin{remark}
  In order to use \eqref{raso} in \eqref{Hardyu},  we should ensure that $\f{2d}{d-2(1-\varepsilon_0)}<d$ which requires $d\geq 4$. For  $d=3$, one can adapt the argument in \cite{TVZ} to bypass the obstacle, we omit the details here.
\end{remark}

 \begin{corollary}\label{hidecay}
  Let $d\geq 4,$ and $u:I\times \mathbb{R}^d \rightarrow \mathbb{C}$ be the spherically symmetric maximal-lifespan solution to \eqref{sch0} which obeys \eqref{assum},\eqref{lowfre1} then
  \begin{align}
    &\lim_{N\rightarrow \infty}[\|u_{\geq N}\|_{\mathcal{S}}+\f{1}{N^{\f{1+\varepsilon}{2}}}\||\nabla|^{-\f{1-\varepsilon}{2}}\nabla u_{<N}\|_{\mathcal{S}}]=0.\label{hidecay}
    \end{align}
    In particular, for any $N>0$ being  a dyadic integer, we have
    \begin{align}
    \|u_{\geq N}\|_{\mathcal{S}}+\f{1}{N^{\f{1+\varepsilon}{2}}}\||\nabla|^{-\f{1-\varepsilon}{2}}\nabla u_{<N}\|_{\mathcal{S}}<\infty ,\quad \text{for all}\quad N>0.  \label{hidecay1}
    \end{align}
  \end{corollary}
  \begin{proof}
    \eqref{hidecay} comes from \eqref{maro21},\eqref{highes} and the scaling invariance of the equation. Now we use \eqref{hidecay} to prove \eqref{hidecay1}.
    Since \eqref{hidecay} implies \eqref{hidecay1} for $N$ is sufficiently large, it suffices to show that \eqref{hidecay1} also holds for $N$ is small.  We may assume $N_0$ such that $N\geq N_0$
    \begin{align}
     \|u_{\geq N}\|_{\mathcal{S}}+\f{1}{N^{\f{1+\varepsilon}{2}}}\||\nabla|^{-\f{1-\varepsilon}{2}}\nabla u_{<N}\|_{\mathcal{S}}<1.
    \end{align}
    For any $N<N_0$, we have
    \begin{align*}
      \|u_{>N}\|_{\mathcal{S}}&=\|u_{\geq N_0}\|_{\mathcal{S}}+\|u_{N<.<N_0}\|_{\mathcal{S}}\\
      &\lesssim 1+\sum_{N< M<N_0}\|u_M\|_{\mathcal{S}}\\
      &\lesssim 1+\sum_{N< M<N_0} M^{-\f{1+\varepsilon}{2}}\||\nabla|^{\f{1+\varepsilon}{2}}u_{M}\|_{\mathcal{S}}\\
      &\lesssim 1+\sum_{N< M<N_0} 1\\
      &<\infty,
    \end{align*}
    and
    \begin{align*}
      \||\nabla|^{\f{1+\varepsilon}{2}}u_{<N}\|_{\mathcal{S}}\lesssim \left(\f{N_0}{N}\right)^{\f{1+\varepsilon}{2}}\f{1}{N_0^{\f{1+\varepsilon}{2}}}\||\nabla|^{\f{1+\varepsilon}{2}}u_{<N_0}\|_{\mathcal{S}}<\infty.
    \end{align*}
    Thus we complete the proof.

  \end{proof}

  \section{the non-evacuation of energy }\label{verification}
In this part, we will prove  that the energy can not evacuate from high frequency to low frequency by showing that $N(t)$ has a lower bound.
\begin{proposition}\label{lowbound}
    Let $d\geq 4,$ and let $u:I\times \R^d \rightarrow \mathbb{C}$ be the critical spherically symmetric  maximal-lifespan solution to \eqref{sch0} which obeys \eqref{assum},\eqref{lowfre1}. Then
    \begin{align}
      \inf_{t\in I} N(t)>0.
    \end{align}
  \end{proposition}

  Assume for contradiction that we have a critical solution $u:I\times \R^d\rightarrow \mathbb{C}$ obeying \eqref{assum} and the hypothesis \eqref{lowfre1} but such that
\[\inf_{t\in I}N(t)=0,\]
 we will obtain the following fact:
\begin{lemma}
 Under the conditions of  Proposition \ref{lowbound}, we have
  \begin{align}\label{lowbound1}
    \limsup_{N\rightarrow\infty }N^{\f{3-\varepsilon}{2}-s_c}\|u_{\geq N}\|_{\mathcal{S}}<\infty.
  \end{align}
\end{lemma}
\begin{proof}
  Let $\eta>0$ be a small number to be chosen later. By \eqref{hidecay}, there exists $\widetilde N_{0}>0$ such that
  \begin{align}
    \|u_{\geq \widetilde N_{0}}\|_{\mathcal{S}}+\f{1}{\widetilde N_{0}^{\f{1+\varepsilon}{2}}}\||\nabla|^{(1+\varepsilon)/2} u_{<\widetilde N_{0}}\|_{\mathcal{S}}\lesssim \eta.
  \end{align}
  By scaling invariance, we may assume $\widetilde N_{0}=1,$ thus
  \begin{align}
    \||\nabla|^{(1+\varepsilon)/2}u_{<1}\|_{\mathcal{S}}&\leq \eta; \label{de1}\\
    \|u_{\geq 1}\|_{\mathcal{S}}&\leq \eta.\label{de2}
  \end{align}
  We claim that:\\
  \begin{claim}
  For any given $\delta>0$ such that
\begin{align}\label{boot1}
  \|u_{\geq N}\|_{\mathcal{S}}\leq \eta N^{-(3-\varepsilon)/2+s_c}+\delta\quad \text{for all}\quad N\geq 1,
\end{align}
 then
  \begin{align}\label{boots}
   \|u_{\geq N}\|_{\mathcal{S}}\leq \eta N^{-(3-\varepsilon)/2+s_c}+\f{\delta}{2}\quad \text{for all}\quad N\geq 1.
  \end{align}
  \end{claim}
 Assuming the  claim,  by iterating the above procedure, we will conclude that
  \begin{align*}
    \|u_{\geq N}\|_{\mathcal{S}}\leq \eta N^{-\f{3-\varepsilon}{2}+s_c}\quad \text{for all}\quad N\geq 1.
  \end{align*}
 Now we are dedicated to proving the claim. Indeed, by choosing  $0<\delta\leq \eta$  such that $\eqref{boot1}$ holds. Furthermore, we can take a dyadic number  $ N_0 \geq 1$   such that $\eta N_0^{-\f{3-\varepsilon}{2}+s_c}\sim \delta,$ then
  \begin{align}\label{boota}
    \|u_{\geq N}\|_{\mathcal{S}}\lesssim \eta N^{-\f{3-\varepsilon}{2}+s_c}\quad \text{for all}\quad 1\leq N\leq   N_0,
  \end{align}
  and
  \begin{align}\label{bootb}
  \|u_{\geq  N_0}\|_{\mathcal{S}}\lesssim \delta.
   \end{align}
  Let $N\geq 1$ , applying $P_{\geq N}$ to both sides of \eqref{sch0} we have
  \begin{align}
    (i\partial_t +\Delta)u_{\geq N}=P_{\geq N}F(u).
  \end{align}
  Hence, by weighted Strichartz estimate \eqref{weightstr} we have
  \begin{align}
    \|u_{\geq N}\|_{\mathcal{S}}\lesssim \||\nabla|^{s_c}u_{\geq N}(t_0)\|_{L_x^2}+\|P_{\geq N}F(u)\|_{\mathcal{N}},
  \end{align}
  for any $t_0 \in I.$   As $\inf_{t\in I}N(t)=0$,  we have
  \begin{align}
    \inf_{t_0 \in I}\||\nabla|^{s_c}u_{\geq N}(t_0)\|_{L_x^2}=0.
  \end{align}
   Thus
   \begin{align}\label{nowaste}
  \|u_{\geq N}\|_{\mathcal{S}}\lesssim \|P_{\geq N}F(u)\|_{\mathcal{N}}.
  \end{align}
  We split $F(u)$ as
  \begin{align*}
    F(u)&=F(u_{<N_0})+O(|u_{\geq N_0}|(|u_{<1}|^{\f{4}{d-2s_c}}+|P_{<1}u_{\leq N_0}|^{\f{4}{d-2s_c}}))\\
    &+O(|u_{\geq N_0}|(|u_{\geq 1}|^{\f{4}{d-2s_c}}+|P_{\geq 1}u_{\leq N_0}|^{\f{4}{d-2s_c}})).
  \end{align*}
  So that we have
  \begin{align}
\text{RHS of \eqref{nowaste}}\lesssim & \|P_{\geq N}O(|u_{\geq N_0}|(|u_{\geq 1}|+|P_{\geq 1}u_{\leq N_0}|)^{\f{4}{d-2s_c}})\|_{\mathcal{N}}\label{nowaste1}\\
&+\|P_{\geq N} O(|u_{\geq N_0}|(|u_{<1}|+|P_{<1}u_{\leq N_0}|)^{\f{4}{d-2s_c}})\|_{\mathcal{N}}\label{nowaste2}\\
&+\|P_{\geq N}F(u_{< N_0})\|_{\mathcal{N}}\label{nowaste3}.
  \end{align}
   By \eqref{nonesti}, \eqref{de2}, \eqref{bootb},we have
 \begin{align*}
   \|P_{\geq N}O(|u_{\geq N_0}||u_{\geq 1}|^{\f{4}{d-2s_c}})\|_{\mathcal{N}}&\lesssim \||u_{\geq N_0}||u_{\geq 1}|^{\f{4}{d-2s_c}}\|_{\mathcal{N}}\\
   &\lesssim \||\nabla|^{s_c}u_{\geq 1}\|_{L_t^\infty L_x^2}^{\f{4}{d-2s_c}}\|u_{\geq N_0}\|_{\mathcal{S}}\\
   &\lesssim \eta^{\f{4}{d-2s_c}}\delta.
 \end{align*}
The other term in  \eqref{nowaste1} is estimated similarly.

For \eqref{nowaste2}, by Lemma \ref{weiber}, \eqref{refine1}  \eqref{de1} and \eqref{bootb} we obtain
\begin{align*}
  &\quad \|P_{\geq N}  O(|u_{\geq N_0}||u_{<1}|^{\f{4}{d-2s_c}})\|_{\mathcal{N}}\\
  &\lesssim \||\nabla|^{\f{1-\varepsilon}{2}-s_c}  O(|u_{\geq N_0}||u_{<1}|^{\f{4}{d-2s_c}})\|_{\mathcal{N}}\\
  &\lesssim  \||\nabla|^{(\f{d}{2}-\f{(d-2s_c)(1+s_c+\varepsilon)}{4})-}u_{<1}\|_{L_t^\infty L_x^2}^{\f{4}{d-2s_c}}\||\nabla|^{(-\f{1-\varepsilon}{2})+}u_{\geq N_0}\|_{\mathcal{S}}\\
  &\lesssim \eta^{\f{4}{d-2s_c}}\delta.
\end{align*}
The other term of \eqref{nowaste2} is estimated similarly.

For the \eqref{nowaste3}, by Lemma \ref{weiber} and \eqref{refine2}
\begin{align*}
\|P_{\geq N}F(u_{<N_0})\|_{\mathcal{N}}&\lesssim  N^{-\f{3-\varepsilon}{2}+s_c}\||x|^{\f{1+\varepsilon}{2}}O (|u_{<N_0}^p \nabla u_{<N_0}|)\|_{L_{t,x}^2}\\
&\lesssim N^{-\f{3-\varepsilon}{2}+s_c} \||\nabla |^{\f{1+\varepsilon}{2}}u_{<N_0}\|_{L_t^\infty L_x^2}^p \||\nabla|^{\f{1-\varepsilon}{2}-s_c+p(s_c-\f{1+\varepsilon}{2})}\nabla u_{<N_0}\|_{\mathcal{S}}.
\end{align*}
Since by \eqref{de1} and \eqref{boota} we have
\begin{align*}
  &\quad \||\nabla|^{\f{1+\varepsilon}{2}} u_{<N_0}\|_{L_t^\infty L_x^2}\lesssim \||\nabla|^{\f{1+\varepsilon}{2}} u_{\leq 1}\|_{L_t^\infty L_x^2}+\sum_{1<M<N_0}\||\nabla|^{\f{1+\varepsilon}{2}} u_M\|_{L_t^\infty L_x^2}\\
  &\lesssim \eta+\eta \sum_{1<M<N_0} M^{\f{1+\varepsilon}{2}} M^{-\f{3-\varepsilon}{2}}\\
  &\lesssim \eta,
\end{align*}
and
\begin{align*}
  &\quad \||\nabla|^{\f{1-\varepsilon}{2}-s_c+p(s_c-\f{1+\varepsilon}{2})}\nabla u_{<N_0}\|_{\mathcal{S}}\\
  &\lesssim  \||\nabla|^{\f{1-\varepsilon}{2}-s_c+p(s_c-\f{1+\varepsilon}{2})}\nabla u_{<1}\|_{\mathcal{S}}+\sum_{1<M< N_0} \||\nabla|^{\f{1-\varepsilon}{2}-s_c+p(s_c-\f{1+\varepsilon}{2})}\nabla u_{M}\|_{\mathcal{S}}\\
  &\lesssim \eta+\eta \sum_{1<M< N_0} M^{\f{1-\varepsilon}{2}-s_c+p(s_c-\f{1+\varepsilon}{2})}M M^{-\f{3-\varepsilon}{2}+s_c}\\
  &\lesssim \eta.
\end{align*}
Thus
\begin{align}
  \|P_{\geq N}F(u_{<N_0})\|_{\mathcal{N}}\lesssim \eta^{p}\eta  N^{-\f{3-\varepsilon}{2}+s_c}.
\end{align}
Combining the separated parts contributed to $\|u_{>N}\|_{\mathcal{S}}$ we have
\begin{align}
  \|u_{>N}\|_{\mathcal{S}}\lesssim \eta^p(\eta N^{-\f{3-\varepsilon}{2}+s_c}+\delta)\quad \text{for} \quad N\geq 1.
\end{align}
By choosing $\eta$ sufficiently small, we complete the proof.
\end{proof}

{\bf Proof of Proposition \ref{lowbound}}
Now we can illuminate that $\inf_{t\in I}N(t)=0$ is incompatible with energy-conservation.
In fact, by \eqref{lowbound1}, for sufficiently large $N$, we have
\begin{align}
\|\nabla P_{N}u\|_{L_{t}^\infty L_x^2}\lesssim N^{-\f{1-\varepsilon}{2}},
\end{align}
and for each  dyadic number $N$
\begin{align}
\|\nabla P_{N}u\|_{L^\infty_t L_x^2}\lesssim N^{1-s_c}.
\end{align}
Thus, by choosing  $M$ sufficiently large
\begin{align*}
 \|\nabla u\|_{L_t^\infty L_x^2} \lesssim \|\nabla u_{<M^{-1}}\|_{L_t^\infty L_x^2}+\|\nabla u_{M^{-1}\leq .< M}\|_{L_t^\infty L_x^2}+\|\nabla u_{\geq M}\|_{L_t^\infty L_x^2},
\end{align*}
as $\inf_{t\in I}N(t)=0$,  we may choose a time sequence $\{t_i\}\in I$ such that  $N(t_i)\rightarrow 0,$  and by dominated convergence theorem we conclude that
\begin{align*}
  \|\nabla u(t_i)\|_{L_x^2}\rightarrow 0 \quad \text{as} \quad N(t_i) \rightarrow 0.
\end{align*}
By interpolation
\begin{align}
\|u\|_{L_x^{p+2}}\lesssim \|u\|_{\f{dp}{2}}^{\theta}\|u\|_{\f{2d}{d-2}}^{1-\theta}\lesssim \||\nabla|^{s_c}u\|_{L_x^2}^{\theta}\|\nabla u\|_{L_x^2}^{1-\theta}\rightarrow 0 \quad \text{as}\quad  N(t_i)  \rightarrow 0,
\end{align}
where $0<\theta<1$.Thus
  \begin{align}\label{energy1}
    E(u)=\int\f{1}{2} |\nabla u|^2 +\f{1}{p+2}|u|^{p+2}dx \rightarrow 0, \quad \text{as} \quad N(t_i) \rightarrow 0.
  \end{align}
 By the energy conservation law of \eqref{sch0}, \eqref{energy1} implies that $u\equiv 0$, which is impossible.
\section{rule out the critical solution  }\label{rule out}
 \begin{theorem}\label{Ibound}
    Let $d\geq 4,$ and let $u:I\times \mathbb{R}^d \rightarrow \mathbb{C}$ be the critical maximal-lifespan spherically symmetric solution to \eqref{sch0} which obeys \eqref{assum},\eqref{lowfre1}. Suppose that $u$ is not identically zero, then $I$ is bounded.
  \end{theorem}
  \begin{proof}
    By \eqref{almost1} and the fact that $N(t)$ has lower bound, we may  choose $N$ sufficiently small such that
    \begin{align}
    \||\nabla|^{s_c}u_{>N}\|_{L_x^{\f{2d}{d-2}}}\gtrsim 1,
    \end{align}
    then integrating with respect to the time variable over the interval $I$, we have
    \[|I|^{\f{1}{2}}\lesssim  \||\nabla|^{s_c}u_{>N}\|_{L_t^2 L_x^{\f{2d}{d-2}}}. \]
    By \eqref{hidecay1} and \eqref{radso}, we know $\||\nabla|^{s_c}u_{>N}\|_{L_t^2 L_x^{\f{2d}{d-2}}}<\infty$,
    which implies  $|I|<\infty$.
    \end{proof}
Theorem \ref{Ibound} means that $u$ blows up in finite time, thus by Corollary \ref{fscale}, $N(t)$ does not  have upper bound in $I$, which is inconsistent with \eqref{lowfre1}.

\section{Appendix}
In this part, we dedicate to proving Lemma \ref{purtabation}. First we  recall the definition of  Strichartz norm and Strichartz estimate.
\begin{definition}[Admissible pair]\label{admi} Let $d\geq 4,$ we call a pair of exponent $(q,r)$ admissible if
\begin{align}
  \frac{2}{q}=d(\f{1}{2}-\f{1}{r}) \quad \textrm{with}\quad 2\leq q\leq \infty.
\end{align}
For a time interval $I$, we define Strichartz norm $\rm{S}(I)$ as 
\begin{align}
  \|u\|_{\rm {S}(I)}:=\sup\{\|u\|_{L^q_tL^r_x(I\times\mathbb{R}^d)}:(q,r)\quad \text{admissible}\}.
\end{align}
We also define the dual of $\rm S(I)$ by $\rm N(I)$, we note that
\begin{align}
  \|u\|_{\rm N(I)}\lesssim \|u\|_{L_t^{q'}L_x^{r'}(I\times \R^d)}\quad \text{for any admissible pair}~(q,r).
\end{align}
\end{definition}
\begin{proposition}[Strichartz estimate]
  Let ~$u:I\times \mathbb{R}^d\rightarrow \mathbb{C}$~be a solution to
  \begin{align}
    (i\partial_t+\Delta)u=F
  \end{align}
  and let~$ s\geq 0$, then
  \begin{align}
    \||\nabla|^su\|_{S(I)}\lesssim\|u(t_0)\|_{\dot{H}^s_x}+\||\nabla|^sF\|_{N(I)},
  \end{align}
  for any ~$t_0 \in I$.
\end{proposition}
In the proof of Lemma \ref{purtabation}, we need the following result. One can carry over the  proof of Lemma 3.4 in \cite{V2}  verbatim.
\begin{lemma}[Persistence of regularity]
  Let $I$ be a compact time interval, and $u$ be  a  solution to \eqref{purtabation1} obeying
  \begin{align}
    \|u\|_{L_{t,x}^{\f{(d+2)p}{2}}(I\times \R^d)}\leq M, \quad \||\nabla|^{s_c}e\|_{\op{N}(I)}\leq L \ll 1,
  \end{align}
  then we have
  \begin{align}
    \||\nabla|^{s_c}u\|_{\rm S(I)}\leq C(M) \|u_0\|_{\dot{H}^{s_c}_x}.
  \end{align}
\end{lemma}

In what follows,  we denote
\begin{align*}
X(I)&:=L_t^{\f{2(d-2s_c+2)}{(d-2s_c)(1-s_c)}}L_x^{\f{2(d-2s_c+2)}{d-2s_c}}(I\times \R^d), \\ Y(I)&:=L_t^\f{2(d-2s_c+2)}{(d+4-2s_c)(1-s_c)}L_x^{\f{2(d-2s_c+2)}{d-2s_c+4}}(I\times \R^d),\\
X'(I)&:=L_t^{\f{2(d-2s_c+2)}{(d-2s_c)(1-s_c)}}\dot{H}^{s_c,\f{2d(d-2s_c+2)}{d^2+4s_c-4s_c^2}}(I\times \R^d),\\
Y'(I)&=L_t^\f{2(d-2s_c+2)}{(d+4-2s_c)(1-s_c)}\dot{H}_x^{s_c,\f{2d(d-2s_c+2)}{d^2+4d-4s_c^2+4s_c}}(I\times \R^d).
\end{align*}
Obviously, we have $X'(I)\hookrightarrow X(I)$, $Y'(I)\hookrightarrow Y(I).$
\begin{remark}
  The reason we choose the particular form of  $X(I)$ and $Y(I)$ stems from the following fact: by dispersive estimate and Hardy-Littlewood-Sobolev inequality we can obtain relatively neat nonlinear estimate
  \begin{align}\label{stris}
    \left\|\int_{0}^t e^{it\Delta}|u|^pu(s)ds \right\|_{X(I)}\lesssim \||u|^pu\|_{Y(I)}\lesssim \|u\|_{X(I)}^{p+1}.
  \end{align}
\end{remark}
Next we will present some nonlinear estimates.
\begin{lemma}\label{dedi}
  Let $F(u)=|u|^pu$ for some $p>0$ and let $0<s<1$. For $1<r,r_1,r_2,\infty$ such that $\f{1}{r}=\f{1}{r_1}+\f{p}{r_2}$, we have
  \begin{align}
    \||\nabla|^s[F(u+v)-F(u)]\|_{L_x^r}\lesssim \||\nabla|^su\|_{L_x^{r_1}}\|v\|_{L_x^{r_2}}^p+\||\nabla|^sv\|_{L_x^{r_1}}\|u+v\|_{L_x^{r_2}}^p.
  \end{align}
\end{lemma}
\begin{lemma}
Let $d\geq 4$, then with spacetime norms over $I\times \R^d$, we have
\begin{align}
  \|F(u)\|_{Y(I)}&\lesssim \|u\|_{X(I)}^{1+p}\label{nona}\\
  \||\nabla|^{s_c}[F(u)-F(v)]\|_{L_t^2 L_x^{\f{2d}{d+2}}}&\lesssim \|u-v\|_{X(I)}^{\f{4}{d-2s_c}}\|v\|_{\dot{S}^{s_c}}+\|u\|_{X(I)}^{\f{4}{d-2s_c}}\|u-v\|_{\dot{S}^{s_c}(I)}\label{nond}
\end{align}
\end{lemma}
\begin{proof}
  \eqref{nona} comes  directly from the definition of $X(I)$ and $Y(I)$. \eqref{nond} from Lemma \ref{dedi}.
\end{proof}

In order to prove Lemma \ref{purtabation}, we primarily establish the short-time perturbation result.
\begin{lemma}[Short-time perturbation]
 Let $d\geq 4,$  $I$ be a compact interval, $\tilde{u}: I\times \R^d\rightarrow \mathbb{C}$ be solution to the equation
  \begin{gather}
  \begin{split}
  \left\{\begin{array}{cc}
    (i\partial_t+\Delta)\tilde{u}=F(\tilde u)+e\\
    \tilde{u}(0)=\tilde{u}_0\in \dot{H}^{s_c}.
    \end{array}\right.
    \end{split}
  \end{gather}
  Suppose
  $$\|\tilde{u}\|_{L_t^\infty \dot{H}^{s_c}(I\times \R^d)}\leq E.$$
  Let $0\in I$ and $u_0 \in \dot{H}^{s_c}_x(\R^d)$. Then there exits $\varepsilon_0, \delta>0$(depending on $E$) with the following properties hold: $0<\varepsilon<\varepsilon_0$, if
  \begin{align}\label{small}
  \|\tilde u\|_{X(I)}\leq \delta,
    \|u_0-\tilde{u}_0\|_{\dot{H}^{s_c}}+\|e\|_{Y'(I)}\leq \varepsilon,
  \end{align}
  then there exits $u: I\times \R^d\rightarrow \mathbb{C}$ solving $(i\partial_t+\Delta)u=|u|^pu \quad \text{with}\quad  u(0)=u_0$. satisfying
  \begin{align}
    \||\nabla|^{s_c}(u-\tilde u)\|_{S(I)}\lesssim \varepsilon^c, \label{non1}\\
    \||\nabla|^{s_c}u\|_{S(I)}\lesssim E,\label{non2}\\
    \||\nabla|^{s_c}(|u|^pu-|\tilde u|\tilde u)\|_{N(I)}\lesssim \varepsilon^c \label{non3}
  \end{align}
  where $c>0$ is a given constant.
\end{lemma}
\begin{proof}
First, we show that $\|u\|_{X(I)}\lesssim  \delta$. Indeed by Duhamel formula\eqref{eq:duhamel}
\begin{align}\label{initial1}
  \|e^{it\Delta}\tilde{u}_0\|_{X(I)}&\lesssim \|\tilde u\|_{X(I)}+\|F(\tilde u)\|_{Y(I)}+\|e\|_{Y'(I)}\\
  &\lesssim \delta+\delta^{1+\f{4}{d-2s_c}}+\varepsilon.
\end{align}
By \eqref{small} and triangle inequality we have
\begin{align}\label{initial2}
  \|e^{it\Delta}u_0\|_{X(I)}\lesssim \delta.
\end{align}
Then using Strichartz estimate \eqref{stris}  and \eqref{nona}, we have
\begin{align*}
  \|u\|_{X(I)}\lesssim  \delta +\|F(u)\|_{Y(I)}\\
  \lesssim  \delta +\|u\|_{X(I)}^{1+\f{4}{d-2s_c}}.
\end{align*}
By continuity  argument we have $\|u\|_{X(I)}\lesssim  \delta$.

We let $w=u-\tilde u$, thus  $w$ satisfies
\begin{align}
  (i\partial_t+\Delta)w=F(u)-F(\tilde u)-e\qquad  w(0)=u_0-\tilde u_0.
\end{align}
By Strichartz estimate \eqref{stris} we have
\begin{align*}
  \|w\|_{X(I)}&\lesssim \|e^{it\Delta}w(0)\|_{X(I)}+\|e\|_{Y(I)}+\|F(u)-F(v)\|_{Y(I)}\\
 & \lesssim \varepsilon +\{\|\tilde u\|_{X(I)}^{\f{4}{d-2s_c}}+\|u\|_{X(I)}^{\f{4}{d-2s_c}}\}\|w\|_{X(I)}\\
 &\lesssim \varepsilon +\delta^{\f{4}{d-2s_c}}\|w\|_{X(I)}.
\end{align*}
Thus by choosing $\delta$ sufficiently small, we have
\begin{align}\label{esww}
\|w\|_{X(I)}\lesssim \varepsilon.
\end{align}
By Strichartz estimate and \eqref{small} and \eqref{nond} we have
\begin{align*}
  \||\nabla|^{s_c}w\|_{S(I)}&\lesssim \|u_0-\tilde u_0\|_{\dot{H}^{s_c}}+\|e\|_{Y'(I)}+\||\nabla|^{s_c}[F(u)-F(\tilde u)]\|_{L_{t,x}^{\f{2(d+2)}{d+4}}}\\
  &\lesssim \varepsilon+\||\nabla|^{s_c}\tilde u\|_{S(I)}\|w\|_{X(I)}^{\f{4}{d-2s_c}}+\||\nabla|^{s_c}w\|_{S(I)}\|u\|_{X(I)}^{\f{4}{d-2s_c}}.
\end{align*}
By \eqref{small} and the persistence of regularity results,  we have $\||\nabla|^{s_c}\tilde u\|_{S(I)}\leq C(\delta)E$.
 For \eqref{non2}, by \eqref{non1} and Strichartz estimate  we have
 \begin{align*}
    &\quad \||\nabla|^{s_c}u\|_{S(I)}\\
   &\lesssim \varepsilon^c+\||\nabla|^{s_c}\tilde u_0\|_{L_x^2}+\|(|\tilde u|^p \tilde u)\|_{Y'(I)}+\|e\|_{Y'(I)}\\
   &\lesssim \varepsilon^c +E+\|\tilde u\|_{X(I)}^{\f{4}{d-2s_c}}\|\tilde u\|_{X'(I)}\\
   &\lesssim \varepsilon^c+ E+C(\delta)E\\
   &\lesssim E.
 \end{align*}
Now \eqref{non3} can be deduced from Lemma \ref{dedi} and \eqref{non1}.
\end{proof}
{\bf Proof of Lemma \ref{purtabation}.}
First note that $\|\tilde {u}\|_{L_{t,x}^{\f{(d+2)p}{2}}(I\times \R^d)}\leq L$, by the persistence of regularity, we have $\|\tilde{u}\|_{X(I)}\lesssim C(E,L)$. Then we may subdivide $I$ into (finitely many, depending on $\delta$ and $L$) intervals $J_k=[t_k, t_{k+1})$ so that
\begin{align}
  \|\tilde u\|_{X(I)}\sim \delta.
\end{align}
then we can use the short-time perturbation results and bootstrap argument to obtain Lemma \ref{purtabation}.

\bibliographystyle{amsplain}

\end{document}